\documentclass[leqno, 12pt, a4paper]{amsart}

\usepackage{amsfonts}
\usepackage{mathrsfs}
\usepackage{amsmath,amssymb,latexsym,amsfonts,amscd}
\usepackage[all,cmtip]{xy}
\usepackage{hyperref}

\usepackage{amssymb}
\usepackage{amsthm}
\usepackage{amsmath}
\usepackage{amsfonts}
\usepackage{amscd}
\usepackage{array}
\usepackage{float}
\usepackage{cite}

\numberwithin{equation}{section}

\theoremstyle{definition}
\newtheorem{defn}{Definition}[section]

\newtheorem{ex}[defn]{Example}

\newtheorem{ques}[defn]{Question}

\theoremstyle{plain}
\newtheorem{prop}[defn]{Proposition}
\newtheorem{conj}[defn]{Conjecture}
\newtheorem{lem}[defn]{Lemma}
\newtheorem{thm}[defn]{Theorem}
\newtheorem{cor}[defn]{Corollary}

\newtheorem{claim}[defn]{Claim}

\theoremstyle{remark}
\newtheorem{rem}[defn]{Remark}
\newtheorem{ack}[defn]{\bf Acknoledgement}

\newcommand{\bQ}{\mathbb{Q}}

\newcommand{\bC}{\mathbb{C}}

\newcommand{\cB}{\mathcal{B}}

\newcommand{\paren}[1]{\left( #1 \right)}

\newcommand{\set}[1]{\left\{ #1 \right\}}
\newcommand{\av}[1]{\left| #1 \right|}

\newcommand{\aw}{\operatorname{aw}}
\newcommand{\lcm}{\operatorname{lcm}}

\begin{document}
\title{On Chern number inequality in dimension 3}
\author{Jheng-Jie Chen}
\address{\rm Department of Mathematics, National Central University, Chungli, 320,
Taiwan}
\email{d94221006@gmail.com}

\maketitle

\begin{abstract}
We prove that if $X\dasharrow X^+$ is a threefold terminal flip, then $c_1(X).c_2(X)\leq c_1(X^+).c_2(X^+)$
where $c_1(X)$ and $c_2(X)$ denote the Chern classes.
This gives the affirmative answer to a Question by Xie \cite{Xie2}.
We obtain the similar but weaker result in the case of divisorial contraction to curves. 
\end{abstract}








\section{Introduction}

The main goals of birational geometry are to classify algebraic varieties up to birational equivalence and to find a good model inside a birational equivalent class. 
Base on the work of Reid, Mori, Koll\'ar, Kawamata, Shokurov, and others, minimal model conjecture in dimension three in characteristic zero was proved by Mori. That is, starting from a mildly singular threefold $X$, there exists a sequence of elementary birational maps (divisorial contractions and flips) such that the end product is either a minimal model or a Mori fiber space.

It is thus natural to expect that further detailed and explicit studies of three dimensional birational maps in minimal model program will be 
useful in the studies of three dimensional geometry in general. The purpose of this article is along this line. Since divisorial contractions to points are intensively studied and classified by Kawamata, Kawakita, Hayakawa and Jungkai Chen, we aim to study the divisorial contraction to curves and flips. More precisely, we compare various invariants of singularties.

This article was motivated by the studies of pseudo-effectiveness of second chern class $c_2(X)$ for terminal threefolds. 
\begin{conj}\label{c2pseff}
Let $X$ be a terminal projective threefold whose anti-canonical divisor $-K_X$ is strictly nef. Then the second Chern class $c_2(X)$ is pseudo-effective.
\end{conj}
Conjecture \ref{c2pseff} is true in the case of the numerical dimension $\nu(-K_X)\neq 2$ due to several works by Miyaoka, Koll\'ar, Mori, Takagi, Keel, Matsuki, McKernan  (cf. \cite{Miya,KMMT,KMM04}).
In the case of numerical dimension $\nu (-K_X)=2$, Conjecture \ref{c2pseff} is true when the irregularity $q(X)\neq 0$ by Xie in \cite{Xie2}.
Furthermore, when $q(X)=0$, Xie considered the following question where the inequality below leads the positive answer to Conjecture \ref{c2pseff}.
\begin{ques}\label{effective}
Let $X_0$ be a $\mathbb{Q}$-factorial projective terminal threefold with $-K_{X_0}$ nef, $\nu(-K_{X_0})=2$, and $q(X_0)=0$.
Suppose $X_0\dashrightarrow X_1\dashrightarrow \cdots\dashrightarrow X_s=Y$ is a composition of divisorial contractions or flips in the minimal model program.
Do we have the inequality $c_1(Y).c_2(Y)\geq 0$?
\end{ques}
Notice that $c_1(X_0).c_2(X_0)\geq 0$ due to Keel, Matsuki and Mckernan \cite[Corollary 6.2]{KMM04}.  
Let $F(X)$ denote the rational number $\sum_i (r_i-1/r_i)$ which is the contribution of non-Gorenstein singularities from the Riemann-Roch formula. 

\begin{thm}{\rm(\cite[Kawamata]{KA3},\cite[Reid]{YPG})}\label{KaRe}
Let $X$ be a projective threefold with at worst canonical singularities.
Then $$\chi (\mathcal{O}_X)=\frac{1}{24}c_1(X).c_2(X)+\frac{1}{24}\sum_i (r_i-1/r_i),$$ where $r_i$ is the index for the virtue singularity $\frac{1}{r_i}(1,-1,b_i)$.
\end{thm}

Xie gave the following more general and interesting questions which are related to Question \ref{effective}. 
\begin{ques}\label{Fflip}
Let $X$ be a $\mathbb{Q}$-factorial terminal projective threefold. Suppose that $X\dashrightarrow X^+$ is a flip.
Can we have the inequality $F(X)\geq F(X^+)$?
\end{ques}
\begin{ques}\label{FtypeI}
Let $X$ be a $\mathbb{Q}$-factorial terminal projective threefold. Suppose that $f\colon X\to Y$ is a divisorial contraction that contracts a divisor to a curve.
Can we have the inequality $F(X)\geq F(Y)$?
\end{ques}

The inequality in Question \ref{Fflip} (resp. Question \ref{FtypeI}) is equivalent to $c_1(X).c_2(X)\leq c_1(X^+).c_2(X^+)$ (resp. $c_1(X).c_2(X)\leq c_1(Y).c_2(Y)$) since $\chi(\mathcal{O}_X)$ is birational invariant. 
It is known that singularities on $X^+$ (resp. $Y$) become better 
by negativity Lemma, 
and both difficulty and depth (cf. \cite[Definition 2.9]{CH11}) decrease (cf. \cite[Proposition 2.1]{CJK13}). 
It is expected to establish the inequalities in Questions \ref{Fflip} and \ref{FtypeI} as well.

The aim of this article is to give the affirmative answer (cf. Theorem \ref{flipEF}) to Question \ref{Fflip}.
Also, we obtain the positive answer (cf. Theorem \ref{typeIEF}) to Question \ref{FtypeI} when $f\colon X\to Y$ is a divisorial irreducible extremal neighborhood (cf. Definition \ref{etnbd}).

We prove these basically by using the classification of extremal neighborhood of Koll\'ar-Mori as follows.
\begin{thm}[Theorem 2.2 in \cite{KM92}]\label{flipnormal}
Suppose $f\colon X\supset C\to Y\ni Q$ is an irreducible extremal neighborhood. Let $E_X\in |-K_X|$ be a general member and $E_Y:=f(E_X)\in |-K_Y|$. 
Then the surfaces $E_X$ and $E_Y$ have at worst Du Val singularities. More precisely, $E_X\to E_Y$ is a partial resolution and every $f,E_X,E_Y$ are classified in Table \ref{ta:table2}.   
\end{thm}
When the extremal neighborhood $f$ is divisorial,
the specific element $E_Y$ yields all possible local general elephants of non-Gorenstein singularity $Q\in Y$ by Table \ref{ta:table1} and Lemma \ref{deform}.
This enables us to compare $F(X)$ and $F(Y)$. 
When $f$ is isolated (that is, a flipping contraction),  we obtain the similar computations by Table \ref{ta:table1}, Lemma \ref{deform} and Theorem \ref{DVelephant}.

In particular, there can't exist a non-Gorenstein singularity of type $cAx/4$, $cD/3$ or $cE/2$ on the contracted curve $\Gamma$ when $f$ is divisorial (resp. on the flipped curve $C^+$ when $f$ is isolated) (cf. Propositions \ref{flipcEcDcAx} and \ref{cEcDcAx}) by using the same idea and some lists given by Koll\'ar and Mori in \cite[Appendix, Theorem 13.17, Theorem 13.18]{KM92}.

\begin{ack}
The author was partially supported by 
NCTS and MOST of Taiwan.
He expresses his gratitude to Professor Jungkai Alfred Chen for extensively helpful and invaluable discussion. 
The author is very grateful to Professor Koll\'ar, Professor Mori and Professor Prokhorov who kindly remind him the easier proof of Theorem \ref{DVelephant} and information for Proposition \ref{k2Asing}.
He would like to thank Professor Kawamata for useful discussion and comments.  
\end{ack}

\section{\bf Preliminaries and notations}

In this section, we recall various notions derived from three dimensional terminal singularities and some basic properties.
We work over complex number field $\mathbb{C}$.

It is known that every terminal 3-fold singularity $P\in X$ is a quotient of isolated compound Du Val singularity by Reid in \cite{Reid83}.
The index of $P\in X$ is defined to be the smallest positive integer $r$ such that $rK_X$ is Cartier at $P$.
In \cite{Mori85}, Mori classified explicitly all such singularities of index $r\geq 2$ which are called non-Gorenstein singularities.
Then, for each non-Gorenstein singularity $P\in X$, the dual graph $\Delta(E)$ of general elephant $E\in | {-} K_X|$ in a neighborhood of $P$ is determined by the following table by Reid in \cite[Section 6]{YPG}.
Here $\textup{aw}$ denotes the axial weight and $ F(X)$ (resp. $\Xi(X)$) denotes the number $\sum_i (r_i-\frac{1}{r_i})$ (resp. $\sum_i r_i$) in Theorem \ref{KaRe}.


{
\begin{table}[H] \label{table1}
\begin{center}
\begin{tabular}{|>{$}c<{$} | >{$}c<{$} >{$}c<{$} >{$}c<{$} >{$}c<{$}  >{$}c<{$} >{$}c<{$} |}
  \hline
  \textup{type}   &\textup{type of action}&  \textup{aw} & \Delta (E) & \textup{basket}& \Xi(P\in X)& F(P\in X)\\ \hline
  cA/r  &\frac{1}{r}(a,-a,1,0)& k&  A_{rk-1} & k \times (b,r) & rk & rk-\frac{k}{r}   \\
  cAx/2  & \frac{1}{2}(0,1,1,1) & 2 &D_{k+2} &2 \times (1,2) & 4&  3  \\
   cAx/4  & \frac{1}{4}(1,1,3,2) & k& D_{2k+1} & \{ (1,4), (k-1) \times (1,2)\} &2k+2& \frac{6k+9}{4}    \\
  cD/2 & \frac{1}{2}(1,0,1,1) & k&D_{2k}&k \times (1,2) & 2k &  \frac{3k}{2}  \\
 cD/3  & \frac{1}{3}(0,2,1,1)& 2&E_6 &2 \times (1,3) & 6 &  \frac{16}{3} \\
  cE/2  & \frac{1}{2}(0,1,1,1)  & 3&E_7&3 \times (1,2) & 6 &  \frac{9}{2} \\ \hline
\end{tabular}
\caption{} \label{ta:table1}
\end{center}
\end{table} \vspace{-0.5truecm}
}

In this article, we fix $X$ to be a $\mathbb{Q}$-factorial projective threefold with at worst terminal singularities and fix $Y$ to be a normal varieity.
Suppose $X\dasharrow Z$ is a birational map where $Z$ is a normal variety.
Let $D$ be a prime divisor on $X$. We denote $D_Z$ the proper transform of $D$ on $Z$.

A birational morphism $f\colon X\to Y$ is called a divisorial contraction to a point $Q$ (resp. a curve $\Gamma$) if the exceptional set $\textup{Exc}(f)=F$ is an irreducible divisor on $X$, relative Picard number  $\rho(X/Y)=1$, $f_*(\mathcal{O}_X)=\mathcal{O}_Y$, and $-K_{X}$ is $f$-ample such that $f(F)$ is a point $Q$ (resp. a curve $\Gamma$).

A birational morphism $f\colon X\to Y$ is called a flipping contraction (resp. flopping contraction) if $\textup{Exc}(f)$ is a curve, $\rho(X/Y)=1$, $f_*(\mathcal{O}_X)=\mathcal{O}_Y$, and $-K_{X}$ is $f$-ample (resp. $f$-trivial).
In this case, the flip (resp. a flop) of $f$ is a birational morphism $f^+\colon X^+\to Y$ where $X^+$ is a $\mathbb{Q}$-factorial projective threefold such that $\textup{Exc}(f^+)$ is a curve, $\rho(X^+/Y)=1$, ${f^+}_*(\mathcal{O}_{X^+})=\mathcal{O}_Y$, and $K_{X^+}$ is $f$-ample (resp. $f^+$-trivial).
$f^+$ is called the flipped contraction (resp. a flopped contraction).
A curve $C$ in the exceptional set $\textup{Exc}(f)$ is called a flipping (resp. flopping) curve.
A curve $C^+$ in the exceptional set $\textup{Exc}(f)$ is called a flipped (resp. flopped) curve.
Note that $C$ (resp. $C^+$) might be reducible.

We recall some definitions in \cite{KM92, CH11}.

\begin{defn} \label{etnbd}
An irreducible extremal neighborhood is a proper bimeromorphic morphism $f\colon X\supset C\to Y\ni Q$ satisfying the following
\begin{enumerate}
\item $X$ is a 3-fold with at worst terminal singularities.
\item $Y$ is normal and $Q$ is the distinguished point.
\item $f^{-1}(Q)=C$ is isomorphic to $\mathbb{P}^1$
\item $K_X\cdot C<0$.

\end{enumerate}

\end{defn}

Let $\Delta(E_X)$ (resp. $\Delta(E_Y)$) denote the dual graph of the general elephant $E_X\in |{-}K_X|$ (resp. $E_Y\in |{-}K_Y|$).
In \cite[Theorem 2.2]{KM92},
Koll\'ar and Mori gave the following explicit list of irreducible extremal neighborhoods. 

{
\small
\begin{table}[H] \label{table2}
\begin{center}
\begin{tabular}{|>{$}c<{$} | >{$}c<{$} >{$}c<{$} >{$}c<{$} >{$}c<{$} >{$}c<{$}|}
  \hline
  \textup{ref in KM} & \textup{type} & \mu_{C \subset X} & \Delta(E_X) & \Delta(E_Y) & \textup{remark} \\ \hline
  2.2.1.1 & cA/m+(III) & m & A_{mk-1} & A_{mk-1} & \\
  2.2.1.2 & cD/3+(III) & 3 & E_6 & E_6 & \\
  2.2.1.3 & IIA(cAx/4)+(III) & 4 & D_{2k+1} & D_{2k+1} & \\
  2.2.1'.1 & cAx/2+(III) & 2 & D_4 & D_4 & \\
  2.2.1'.2 & cD/2+(III) & 2 & D_{2k} & D_{2k} & \\
  2.2.1'.3 & cE/2+(III) & 2 & E_7 & E_7 &  \\
  2.2.1'.4 & IIA(cAx/4)+(III) & 4 & D_{2k+1} & D_{2k+1} & \\
  2.2.2  & IC(quot) & m & A_{m-1} & D_m & m \text{ is odd } \\
  2.2.2'  & IIB & 4 & D_5 & E_6 &    \\
  2.2.3 & IA+IA & m& A_{m-1}+D_{2k} & D_{2k+m} & m \textup{ is odd} \\
  2.2.3' & IA+IA+III & m&  A_{m-1}+A_1 & D_{m+2} & m \textup{ is odd} \\
  2.2.4 & ss IA+IA & \max\{r_1, r_2\} & A_{r_1 k_1 -1}+A_{r_2 k_2-1} &
    A_{r_1 k_1+r_2 k_2-1} & \\
  2.2.5 & \textup{Gorenstein} & 1 & \textup{smooth} &
    \textup{smooth} & \\
    \hline
\end{tabular}
\caption{} \label{ta:table2}
\end{center}
\end{table} \vspace{-0.5truecm}
}
Note that $\Delta (E_Y)$ is $A$-type only in cases 2.2.1.1 and 2.2.4 which are defined to be semistable extremal neighborhood.
The extremal neighborhood $X\supset C$ is called isolated if $f|_{X-C}\colon X-C\to Y-\{Q\}$ is an isomorphism (cf. Remark \ref{isolated}). Otherwise, it is called divisorial.

If $f$ is divisorial, 
we define
$$ \left\{ \begin{array}{ll}
r_X=r_{ C \subset X} &=  \text{lcm} \{\textup{ index }r(P) | P \in C \}; \\
\mu_X=\mu_{C \subset X} &=\max\{\textup{ index } r(P) | P \in C\}; 
\end{array} \right.$$
 

Similarly, if $f \colon X \to Y$ is a flipping contraction and $f^+ \colon X^+ \to Y$ is the flip,
we define $\mu_{X^+}$ (resp. $r_{X^+}$) to be the maximum (resp. the least common multiple) of indices of singularities on flipped curves $C^+$.

\begin{defn} \label{wmor}
Suppose $P\in X$ is a terminal 3-fold singularity with index $r>1$.
We say that $g\colon W\supset G\to X\ni P$ is a $w$-morphism if it is a divisorial contraction that contracts the divisor $G$ to the point $P$ with minimal discrepancy $a(G,X)=1/r$.
\end{defn}

\subsection{Cartier index} 

\noindent

In this subsection, we collect some known results.


\begin{lem} \label{comp_d}
Let $f \colon X \to Y$ be a divisorial contraction that contracts the divisor $F$ to a curve $\Gamma$. If $Q\in \Gamma$ has index $r$, we have $r \mid r_X$ and $2 r \leq \mu_X$.
\end{lem}
\begin{proof}
Let $g \colon W \to X $ be a resolution of $X$ obtained by successive weighted blowups over singular points on $f^{-1}(\Gamma)$. Then we may write
\[
  K_W = g^*K_X + \sum_{i=1}^s \frac{a_{i}}{r_i} F_i \quad \textrm{and} \quad
  g^*F = F_W + \sum_{i=1}^s \frac{\alpha_i}{r_i} F_i,
\]
where all the integer $a_i > 0$ and $\alpha_i > 0$. Therefore,
\[
  K_W = g^*f^*K_Y + F_W + \sum_{i=1}^s \frac{a_i+\alpha_i}{r_i} F_i.
\]

Now $f \circ g : W \to Y$ is a resolution of $Y$ as well. There must exist an exceptional divisor over $Y$ with discrepancy $\frac{1}{r}$ by \cite{Haya99, Haya00}.
Hence for some $i$, we have
\[
  \frac{1}{r_Y}
  = \frac{a_i+\alpha_i}{r_i}
  \geq \frac{2}{r_i}
  \geq \frac{2}{\mu_X}. \qedhere
\]
\end{proof}

\begin{lem} \label{comp_f}
Let $f \colon X \to Y$ be a flipping contraction and $f^+ \colon X^+ \to Y$ be the flip. Then $\mu_{X^+} \leq \mu_X$.
\end{lem}
\begin{proof}
Let $W$ be a common resolution of $X$ and $X^+$ and let $g \colon W \to X$ and $g^+ \colon W \to X^+$ be corresponding morphisms.
Then we may write
\[
  K_W
  = g^*K_X + \sum_{i=1}^s \frac{a_{i}}{r_i} F_i
  = {g^+}^*K_{X^+} + \sum_{i=1}^s b_{i} F_i
\]
with $\frac{a_i}{r_i} \leq b_i$.

There must exist an exceptional divisor over $X^+$ with discrepancy $\frac{1}{\mu_{X^+}}$. Hence $\frac{1}{\mu_{X^+}} = b_i $ for some $i$ and it follows that
\[
  \frac{1}{\mu_{X^+}}
  = b_i
  \geq \frac{a_i}{r_i}
  \geq \frac{1}{r_i}
  \geq \frac{1}{\mu_X}. \qedhere
\]
\end{proof}
From Lemma \ref{comp_d} and \cite[Theorem 4.2]{KM92}, we have the following assertion.
\begin{cor} \label{exclude}
Let $f \colon X \to Y$ be a divisorial contraction to a curve $\Gamma$ with $\mu_X \leq 3$. Then $Y$ has only Gorenstein singularities. Similarly, if $ f \colon X \to Y$ is a flipping contraction with $\mu_X \leq 2$, then $X^+$ has only Gorenstein singularities.
\end{cor}

\begin{lem} \label{deform}
Let $P \in X$ be a terminal singularity and let $D \in \av{-K_X}$ be an irreducible element. Suppose that $D$ is of type $E_n$ then the general elephant is of type $E_m$, $D_m$, or $A_m$ with $m \leq n$ (equality holds only when $E_m = E_n$). Similarly, if $D$ is of type $D_n$, then the general elephant is of type $D_m$, or $A_m$ with $m \leq n$ (equality holds only when $D_m = D_n$). Also, if $D$ is of type $A_n$, then the general elephant is of type $A_m$ with $m \leq n$.
\end{lem}
\begin{proof}
This is the case since corank and milnor number are semicontinuous. See \cite[Corollary 2.49, 2.52, 2.54]{deformation} for details.
\end{proof}

\section{Flipping contraction}

In this section, we prove the inequality $F(X) \geq F(X^+)$ for any threefold terminal flip $X \dasharrow X^+$ (cf. Theorem \ref{flipEF}).
Notice that the flipping curve $C\subset X$ can be assumed to be irreducible by \cite[Section 8]{Kaw88} or \cite[Theorem 2.3]{CH11}.
We follow the classification of extremal neighborhood of Koll\'{a}r-Mori in Table \ref{ta:table3}. 
\begin{rem}\label{isolated}

By \cite[Theorem 2.2]{KM92}, the isolated extremal neighborhoods are classified by the following.

{\tiny
\begin{table}[H] \label{table3}
\begin{center}
\begin{tabular}{|>{$}c<{$} | >{$}c<{$} >{$}c<{$} >{$}c<{$} >{$}c<{$} >{$}c<{$}|}
  \hline
  \textup{ref in KM} & \textup{type} & \mu_{C \subset Y} & \Delta(E_Y) & \Delta(E_X) & \textup{remark} \\ \hline
  2.2.1.1 & cA/m & m & A_{mk-1} & A_{mk-1} & \\
  2.2.1.2 & cD/3 & 3 & E_6 & E_6 & \\
  2.2.1.3 & IIA(cAx/4) & 4 & D_{2k+1} & D_{2k+1} & \\
  2.2.2  & IC(quot) & m & A_{m-1} & D_m & m \text{ is odd } \\
  2.2.3 & IA+IA & m& A_{m-1}+D_{2k} & D_{2k+m} & m \textup{ is odd} \\
  2.2.4 & ss IA+IA & \max\set{r_1, r_2} & A_{k_1 r_1-1}+A_{k_2r_2-1} &
    A_{k_1r_1+k_2r_2-1} & \\ \hline
\end{tabular}
\caption{} \label{ta:table3}
\end{center}
\end{table} \vspace{-0.5truecm}
}
\end{rem}

We start with the useful result which can be viewed as an application of Theorem \ref{flipnormal}. 
\begin{thm}\label{DVelephant}
Suppose that $f \colon X \to Y$ is a flipping contraction and $f^+ \colon X^+ \to Y$ is the flipped contraction of $f$.
Let $E_{X}\in | {-}K_X|$ be a general element and $E_{X^+}\in |{-}K_{X^+}|$ be its proper transform.
Then $E_{X^+}$ is normal near the flipped curve and has at worst Du Val singularities. In particular, if $S$ is the minimal resolution of $E_X$, then $S$ dominates $E_{X^+}$.
\end{thm}

\begin{proof}
By  \cite[Corollary 5.25]{KM98}, the surface $E_{X^+}$ is
Cohen-Macaulay since $X^+$ has at worst terminal singularities and $E_{X^+}$ is a $\mathbb{Q}$-Cartier Weil divisor on $X^+$. Hence $S_2$ is satisfied.

The surfaces $E_X$ and $E_Y$ are normal and have at worst Du Val singularities and the restriction morphism $E_X\to E_Y$ is crepant by Theorem \ref{flipnormal}.
By inverse of adjunction, the pair $(X,E_X)$ is canonical.
Since $K_X+E_X=\mathcal{O}_{X}$ is $f$-trivial, the pairs $(X^+,E_{X^+})$ and $(X,E_X)$ have the same singularities. Let $g\colon W\to X^+$ be the blowup along an irreducible component $\Gamma$ of the flipped curve $C^+$. 
Since $(X^+,E_{X^+})$ is canonical and $C^+$ is contained in $E_{X^+}$, we see that $K_{W}+E_W=g^*(K_{X^+}+E_{X^+})$. In particular, $E_{X^+}$ is $R_1$ near $C^+$, and hence the surface $E_{X^+}$ is normal.

Because $K_{E_{X^+}}=\mathcal{O}_{E_{X^+}}$ is $f^+|_{E_{X^+}}$-trivial, the restriction morphisms $E_X\to E_Y$ and $f^+|_{E_{X^+}}\colon E_{X^+}\to E_Y$ are both crepant. Hence $S$ is also the minimal resolution of $E_Y$.
\end{proof}

According to Theorem \ref{DVelephant}, we are able to exclude some non-Gorenstein singularity types on the flipped curve $C^+$.

\begin{prop}\label{flipcEcDcAx}
Let $f \colon X \to Y$ be an irreducible flipping contraction and let $f^+ \colon X^+ \to Y$ be the flip of $f$.
If $P'\in C^+\subset X^+$ is a non-Gorenstein singularity, then $P'$ can not be of type $cE/2$,  $cD/3$ nor $cAx/4$.
\end{prop}

\begin{proof}
If $P' \in X^+$ is of type $cE/2$ (resp. $cD/3$), then dual graph of general elephant of $P\in X^+$ is of type $E_7$ (resp. $E_6$) by Table \ref{ta:table1}.
As $C^+$ corresponds to one vertex of dual graph $\Delta (E_Y)$, the dual graph $\Delta (E_{X^+})$ is better than $\Delta (E_Y)$ which is at worst $E_6$ from the descriptions in Table \ref{ta:table2}. Hence $P'\not\in C^+$ by Lemma \ref{deform}.
In cases 2.2.1.2, 2.2.1.3, 2.2.2 and 2.2.3,
every non-Gorenstein singularity on the flipped curve $C^+$ is of index $2$ or $3$ by \cite[Theorem 13.17, Theorem 13.18]{KM92}.
So $C^+$ cannot contain singularities of type $cAx/4$ by Remark \ref{isolated} and Remark \ref{cAflip} below.
\end{proof}

\begin{rem}\label{cAflip}
Suppose $P\in C^+$ is a non-Gorenstein singularity of $X^+$.
In the semistable cases (That is, in cases 2.2.1.1 and 2.2.4), the dual graph $\Delta (E_Y)$ is $A$-type,
so is each connected component of $\Delta(E_{X^+})$. In particular, $P\in X^+$ is of type $cA/r$ by Lemma \ref{deform}.

Notice that there are at most two connected components of $\Delta(E_{X^+})$ in the semistable cases since $C^+$ corresponds one vertex of dual graph
$\Delta (E_Y)$.
Therefore, 
the normal surface $E_{X^+}$ contains at most two singularities near $C^+$ by contracting exceptional curves in the minimal resolution of $E_{X^+}$.
This implies that $X^+$ contains at most two non-Gorenstein singularities on $C^+$ by Lemma \ref{deform}.
\end{rem}

Moreover, in the case 2.2.4, the singularities on the flipped curve $C^+$ are classified by Mori. 
\begin{prop}[Mori]\label{k2Asing}
Suppose $X\supset C$ is in the case $2.2.4$ and $C$ is a flipping curve. Let the singularities on $C$ be of types $cA/r_1$ and $cA/r_2$ with axial weights $k_1$ and $k_2$ as in Table \ref{table2}.
Then the flipped curve $C^+$ contains exactly two singularities $cA/r_1'$ and $cA/r_2'$ with axial weights $k_1'$ and $k_2'$.
Furthermore, by rearranging the indices, we have $r_1\geq r_1', r_2\geq r_2', k_1\leq k_1',$ and $k_2\leq k_2'$. 
\end{prop}
\begin{proof}
The first assertion follows from \cite[Theorem 4.7]{Mori02}. We adopt the notations in \cite{Mori02} to prove the inequalities. 
Put $d(i)=m_i=r_i$ and $\alpha_i=k_i$ for $i=1,2$. From \cite[Definition 3.2]{Mori02}, Mori defined the sequence $d(n)\in \mathbb{Z}, n\in \mathbb{Z}$ by 
$$d(n+1)+d(n-1)=\delta \rho_n d(n).$$
From \cite[Corollary 4.1, Definition 4.2, Theorem 4.7]{Mori02},
there exists a positive integer $k\geq 3$ with the indices $r_1'=m_1'=d(k-1)>0$ and $r_2'=m_2'=-d(k)>0$ and the corresponding axial weights $\alpha_{k-1}+\rho_{k-1}e(k+1)$ and $\alpha_{k-2}+\rho_{k-2}e(k)$ where each $\rho_{j},\alpha_{j}\in \mathbb{N}$ and $e(j)\in \mathbb{Z}$ are defined in \cite[Definition 3.2]{Mori02}.
From \cite[Lemma 3.3.1, Corollary 3.4, Lemma 3.5]{Mori02}, it follows that 
\begin{align*}
m_1'&=d(k-1)<d(k-3)<\cdots < d(1) \textup{ or } d(2)\textup{ and }\\
m_2'&=-d(k)=d(k-2)-\delta \rho_{k-1}d(k-1)\\&<d(k-2)<d(k-4)<\cdots< d(2) \textup{ or } d(1). \\
\end{align*} 
By exchanging $r_1'$ and $r_2'$ (resp. $k_1'$ and $k_2'$), we may assume that $m_1'\leq d(1)$ and $m_2'<d(2)$.
Now $e(k),e(k+1)>0$ if $k\geq 4$ by \cite[Corollary 3.8]{Mori02}.
So the above axial weights are greater or equal to $\alpha_{k-1}, \alpha_{k-2}$ respectively.
From \cite[Definition 3.2]{Mori02}, $\alpha_3\geq \alpha_1$ and $\alpha_4\geq \alpha_2$ and $\alpha_i=\alpha_{i+4j}$ for every positive integer $j$.
In particular, $\alpha_{k-1}\geq \alpha_{1}$ and $\alpha_{k-2}\geq \alpha_{2}$. 
\end{proof}

Now, we give the affirmative answer to Question \ref{Fflip}.
\begin{thm}\label{flipEF}
Let  $f\colon X \to Y$ be a flipping contraction and $f^+\colon X^+\to Y$ be the corresponding flipped contraction. 
Then $\Xi (X)\geq \Xi(X^+)$ and $F(X)\geq F(X^+)$, where $\Xi(X)$ denotes the integer $\sum_i r_i$ in Theorem \ref{KaRe}.
\end{thm}
\begin{proof}
We first deal with the non-semistable extremal neighborhood. 
Let $E_{X^+}\in |-K_{X^+}|$ be a general elephant.

\noindent
{\bf Case 1.} (In case $2.2.1.2$) $cD/3$, $\Delta(E_X)=\Delta(E_Y)=E_{6}$.
There are at most one singularity $P'$ of index 2 on $C^+$ by \cite[Chap 13, Appendix]{KM92}.
Let $k'$ be the axial weight of $P'$. 
Since $E_{X^+}$ is a partial resolution of $E_Y$ and $\Delta(E_Y)=E_{6}$, by Lemma \ref{deform},
the general elephant of $P'\in X^+$ is $A_{n_1}$ or $D_{n_2}$ with $n_1\leq 5$ and $n_2\leq 5$.
By Proposition \ref{flipcEcDcAx}, $P'\in X^+$ is of type $cA/2$, $cAx/2$ or $cD/2$.

\noindent
{\bf 1.1.} The index $2$ point $P'$ on $C^+$ is of type $cA/2$.
We have $2k'-1\leq 5$, so $\Xi(X^+)=2k'\leq 6=\Xi(X)$ and $F(X^+)= \frac{3k'}{2}\leq \frac{18}{4} < \frac{16}{3}=F(X)$.

\noindent
{\bf 1.2.} The index $2$ point $P'$ on $C^+$ is of type $cAx/2$. \\
We have $\Xi(X^+)=4< 6=\Xi(X)$ and $F(X^+)= 3<\frac{16}{3}=F(X)$.

\noindent
{\bf 1.3.} The index $2$ point $P'$ on $C^+$ is of type $cD/2$. \\
We have $2k'\leq 5$, so $\Xi(X^+)=2k'< 6=\Xi(X)$ and $F(X^+)= \frac{3k'}{2}\leq \frac{15}{4} <\frac{16}{3}=F(X)$.

\bigskip
\noindent
{\bf Case 2.} (cf. [2.2.1.3] of \cite{KM92}) $IIA$ ($cAx/4$), $\Delta(E_X) = \Delta(E_Y) = D_{2k+1}$.
There is one singularity $P_1'$ of index 2 and probably another singularity $P_2'$ of index $3$ on $C^+$ by \cite[Chap 13, Appendix]{KM92}.
Let $k_1'$ and $k_2'$ be the corresponding axial weights of $P_1'$ and $P_2'$, respectively.
By Proposition \ref{flipcEcDcAx}, $P_1'\in X^+$ is of type $cA/2$ or $cAx/2$ or $cD/2$ and $P_2'\in X^+$ is of type $cA/3$.

\noindent
{\bf Subase 2-1.} Suppose there is only one non-Gorenstein singularity $P_1'\in C^+$.

\noindent
{\bf 2-1.1.} The index $2$ point $P_1'$ on $C^+$ is of type $cA/2$.
We have $2k_1'-1\leq 2k$, so $\Xi(X^+)=2k_1'< 2k+2 =\Xi(X)$ and $F(X^+)= \frac{3k_1'}{2} < \frac{6k+9}{4}=F(X)$.

\noindent
{\bf 2-1.2.} The index $2$ point $P_1'$ on $C^+$ is of type $cAx/2$. \\
We have $\Xi(X^+)=4< 2k+2=\Xi(X)$ and $F(X^+)= 3< \frac{6k+9}{4}=F(X)$.

\noindent
{\bf 2-1.3.} The index $2$ point $P_1'$ on $C^+$ is of type $cD/2$. \\
We have $2k_1'\leq 2k$, so $\Xi(X^+)=2k_1'< 2k+2=\Xi(X)$ and $F(X^+)= \frac{3k_1'}{2}< \frac{6k+9}{4}=F(X)$.

\noindent
{\bf Subase 2-2.} Suppose there are two non-Gorenstein singularities $P_1',P_2'\in C^+$.
The index $3$ singularity $P_2'$ is $cA/3$ from Proposition \ref{flipcEcDcAx}.
By classification in \cite[Appendix A.2]{KM92}, the extremal neighborhood $X\supset C\ni P$ is in \cite[Appendix (A.2.2.1)]{KM92}.
That is, $$(X,P)=(y_1,y_2,y_3,y_4 ;\alpha)/\mathbb{Z}_4(1,1,3,2;2)\supset C=y_1\textup{-axis}/\mathbb{Z}_4, $$
$$\textup{and }\alpha=0\cdot y_4+y_3^2+g(y_1,y_2)y_2+\cdots\in (y_2,\ y_3,\ y_4)$$
where $g(y_1,y_2)$ is a nonzero linear form in $y_1,y_2$ with the condition
$$\alpha\equiv y_1y_2\textup{\ \ \ \ mod\ }(y_2,\ y_3,\ y_4)^2.$$
By a coordinate change in $y_1, y_2$, we may assume that $\alpha=y_1^2+y_3^2+f(y_2,y_4)$ where $y_2^2$ appears in $f(y_2, y_4)$.
If we put $\tau-wt(y_2)=1/4,$ and $\tau-wt(y_4)=2/4$, then $\tau-wt(f(y_2,y_4))=\tau-wt(y_2^2)=1/2$.

\begin{claim}
$k_2'\leq 2.$
\end{claim}
\begin{proof}
From \cite[Theorem 3.3]{CH11}, the flip $X\dasharrow X^+$ can be factored into the diagram.
\[
\xymatrix{
W\ar@{-->}[rr]^{h} \ar[d]_g & &W'\ar[d]^{g'} \\
X\ar@{-->}[rr]\ar[rd]_f&& X^+\ar[ld]^{f^+}\\
&Y}
\]
where $t$ is a positive integer, $g$ is a $w$-morphism, 
$g'$ is a divisorial contraction, and $h$ is a composition of flips and probably a flop.
By \cite[Theorem 7.4, Theorem 7.9]{Haya99}, the $w$-morphism $g\colon W\to X$ with center $P$ is actually the weighted blowup with weight
$$wt(y_1,y_2,y_3,y_4)=(\frac{1}{4},\frac{1}{4},\frac{3}{4},\frac{2}{4})\textup{ or }  (\frac{5}{4},\frac{1}{4},\frac{3}{4},\frac{2}{4}),$$
and the non-Gorenstein singularities on $W$ consist of a cyclic quotient point of index $\leq 5$ and at worst a point $cD/2$ with axial weight $k''$.
By Lemma \ref{comp_f}, we see that the maximum index of $W'$ $\leq 5$. Since $P_2'$ is a point of index $3$ in the flipped curve $C^+$, by Lemma \ref{comp_d}, the divisorial contraction $g'$ must contract a divisor to a point.
Denote by $G$ the exceptional divisor of $g$.
For our purpose, we may assume that the center of $g'$ has index $3$. So the center of $g'$ must be $P_2'$.
From Kawakita's classification in \cite[Theorem 1.2]{Kawakita2}, $g'$ is a weighted blow up 
and there exists at most three non-Gorenstein singularities on the exceptional divisor $G_{W'}$ where at most one singularity is not cyclic quotient.

For each projective terminal threefold $Z$, we define $$\Xi_{>2}(Z):=\sum_{i, \textup{ where }r_i>2} k_ir_i.$$
Suppose $W_j\dasharrow W_{j+1}$ is a flip which factors through $h\colon W\dasharrow W'$. Let $W_j\supset C_j$ be an isolated irreducible extremal neighborhood.
If $W_j\supset C_j$ is in the case 2.2.1.1, by Lemma \ref{comp_f} and $\Xi(W_j)\geq \Xi(W_{j+1})$ in case 5 of the proof in this Theorem, we see $\Xi_{>2}(W_j)\geq \Xi_{>2}(W_{j+1})$.
If $W_j\supset C_j$ is in one of cases 2.2.2 and 2.2.3 in Table \ref{ta:table3}, there is no singularity of index $\geq 3$ on the flipped curve, so $\Xi_{>2}(W_j)\geq \Xi_{>2}(W_{j+1})$. Notice that $C_j$ contains no point of type $cAx/4$ and $cD/3$ (resp. $cD/2$) of $W_j$ by Proposition \ref{flipcEcDcAx} (resp. by Table \ref{ta:table3} and \cite[Remark 1]{Mori07}). In particular, $W$ is isomorphic to $W_j$ in an open neighborhood of the singularity $cD/2$.
When $W_j\supset C_j$ is in the case 2.2.4, by Case 6 of the proof in this Theorem, we see that $\Xi(W_{j})\geq \Xi(W_{j+1})$ and hence
$$5\geq \Xi_{>2}(W)=\Xi(W)-2k''\geq \Xi(W_{j})-2k''\geq \Xi(W_{j+1})-2k''\geq  \Xi_{>2}(W_{j+1}).$$   
In all cases, we have 
$$3k_2'=\Xi_{>2}(X^+)\leq \Xi_{>2}(W')+2\leq \Xi_{>2}(W)+2\leq 5+2=7.$$
This implies $k_2'\leq 2$.

\end{proof}
Therefore $2k_1'+3k_2'\leq 2k+1$ and so $\Xi(X^+)=2k_1'+3k_2'<2k+2 =\Xi(X)$ and
$$F(X^+)= \frac{3k_1'}{2} +\frac{8k_2'}{3}\leq \frac{3k_1'}{2} +\frac{16}{3}<\frac{6k+9}{4}=F(X).$$

\noindent
{\bf Case 3.} (In case $2.2.2$) $IC$ (cyclic quotient), $\Delta(E_X) = A_{m-1}, \Delta(E_Y) = D_{m}$ and $m\geq 5$ is odd.
There are at most one singularity $P'$ of index 2 on $C^+$ by \cite[Chap 13, Appendix]{KM92}. Let $k'$ be the axial weight of $P'$.\\\

\noindent{\bf 3.1.} The index $2$ point $P'$ on $C^+$ is of type $cA/2$.
We have $2k'-1 \leq  m-1$. Since $m$ is odd, we see $\Xi(X^+)=2k'< m=\Xi(X)$ and $F(X^+)= \frac{3k'}{2}< m - \frac{1}{m}=F(X)$.

\noindent
{\bf 3.2.} The index $2$ point $P'$ on $C^+$ is of type $cAx/2$. \\
We have $\Xi(X^+)=4< m=\Xi(X)$ and $F(X^+)= 3< m - \frac{1}{m}=F(X)$.

\noindent
{\bf 3.3.} The index $2$ point $P'$ on $C^+$ is of type $cD/2$. \\
We have $2k'\leq m-1$, so $\Xi(X^+)=2k'< m=\Xi(X)$ and $F(X^+)= \frac{3k'}{2}< m - \frac{1}{m}=F(X)$.

\noindent
{\bf Case 4.} (In case $2.2.3$) $IA+IA$, $\Delta(E_X) = A_{m-1}+D_{2k}, \Delta(E_Y) = D_{m+2k}$. Note that $m\geq 5$ is odd.
There are at most one singularity $P'$ of index 2 on $C^+$ by \cite[Chap 13, Appendix]{KM92}. Let $k'$ be the axial weight of $P'$.\\

\noindent
{\bf 4.1.} The index $2$ point $P'$ on $C^+$ is of type $cA/2$. \\
We have $2k'-1<m+2k$, so $\Xi(X^+)=2k'< m+2k=\Xi(X)$ and $F(X^+)= \frac{3k'}{2}< m - \frac{1}{m} + \frac{3k}{2}=F(X)$.

\noindent
{\bf 4.2.} The index $2$ point $P'$ on $C^+$ is of type $cAx/2$. \\
We have $\Xi(X^+)=4< m+2k=\Xi(X)$ and $F(X^+)= 3< m - \frac{1}{m} + \frac{3k}{2}=F(X)$.

\noindent
{\bf 4.3.} The index $2$ point $P'$ on $C^+$ is of type $cD/2$. \\
We have $2k'< m+2k$, so $\Xi(X^+)=2k'< m+2k=\Xi(X)$ and $F(X^+)= \frac{3k'}{2}< m - \frac{1}{m} + \frac{3k}{2}=F(X)$.

\noindent
{\bf Case 5.} (In case $2.2.1.1$) $cA/m$, $\Delta(E_X) =\Delta(E_Y)= A_{rk-1}$. \\
By Remark \ref{cAflip}, there are at most two non-Gorenstein singularities $P'_1,P_2'$ on $C^+$ and each $P_i'$ is also of type $cA/r_i'$. For $i=1,2$, let $k_i'$ be the axial weight for the point $P_i'$.
We have $\Xi(X^+)=r_1'k_1'+r_2'k_2'\leq rk=\Xi(X)$. By Lemma \ref{comp_f}, $r_i'\leq r$ for $i=1,2$.

\noindent
{\bf 5.1.}
Suppose that $k> k_1'+k_2'$. 
Then
\begin{align*}
&F(X^+) - F(X)=k_1' \paren{r_1'-\frac{1}{r_1'}}+k_2' \paren{r_2'-\frac{1}{r_2'}} - k \paren{r-\frac{1}{r}}\\
&<k_1' \paren{r_1'-\frac{1}{r_1'}}+k_2' \paren{r_2'-\frac{1}{r_2'}} - \paren{k_1'+k_2'} \paren{r-\frac{1}{r}}\\
&=k_1' \paren{r_1'-\frac{1}{r_1'}-r+\frac{1}{r}}+k_2' \paren{r_2'-\frac{1}{r_2'}-r+\frac{1}{r}}\leq 0.
\end{align*}

\noindent
{\bf 5.2.}
Suppose that $k\leq k_1'+k_2'$. Then $r_1'r_2'k\leq r_1'r_2'(k_1'+k_2')\leq rr_2'k_1'+r_1'rk_2'$.
Together with $r_1'k_1'+r_2'k_2'\leq rk$, we obtain $$F(X^+)-F(X)=r_1'k_1'+r_2'k_2'-rk+\paren{\frac{k}{r}-\frac{k_1'}{r_1'}-\frac{k_2'}{r_2'}}\leq 0.$$


\noindent
{\bf Case 6.} (In case $2.2.4$) semistable $IA+IA$, $\Delta(E_X) = A_{r_1k_1-1}+A_{r_2k_2-1}, \Delta(E_Y) = A_{r_1k_1+r_2k_2-1}$. \\
We have $\Xi(X^+) =r_1'k_1'+r_2'k_2' \leq r_1k_1+r_2k_2=\Xi(X)$.
From Proposition \ref{k2Asing}, $C^+$ contains exactly two singularities $cA/r_1'$ and $cA/r_2'$ with axial weights $k_1'$ and $k_2'$ such that  $r_1\geq r_1', r_2\geq r_2', k_1\leq k_1',$ and $k_2\leq k_2'$.
Note that either $r_1> r_1'$ or $r_2> r_2'$ from the proof of Proposition \ref{k2Asing}.
So $r_1'k_1\leq r_1k_1'$ and $r_2'k_2\leq r_2k_2'$ and
\begin{align*}
&F(X^+) - F(X)\\
&=k_1' \paren{r_1'-\frac{1}{r_1'}}+ k_2' \paren{r_2'-\frac{1}{r_2'}}- k_1 \paren{r_1-\frac{1}{r_1}}- k_2 \paren{r_2-\frac{1}{r_2}}\\
&=\paren{r_1'k_1'+r_2'k_2'-r_1k_1-r_2k_2}+\paren{\frac{k_1}{r_1}+\frac{k_2}{r_2}-\frac{k_1'}{r_1'}-\frac{k_2'}{r_2'}}< 0.
\end{align*}


\begin{rem}
From above computations, we observe the strict inequality $F(X)>F(X^+)$ except the case when the extremal neighborhood $f\colon X\supset C$ is in $2.2.1.1$, $r=r_1'=r_2'$ and $k=k_1'+k_2'$.
\end{rem}

\end{proof}

\section{Divisorial irreducible extremal neighborhood}
In this section, we fix $f\colon X\supset C\to Y\ni Q$ to be an irreducible extremal neighborhood that contracts a divisor to a curve $\Gamma$ as in Definition \ref{etnbd}.
The purpose is to prove that $F(X) > F(Y)$. 
By Lemma \ref{comp_d}, we consider those cases with $\mu_{X \supset C} \geq 4$ only. 

We begin with the following observation which is similar to Proposition \ref{flipcEcDcAx}.
\begin{prop}\label{cEcDcAx}
Let $f\colon X\supset C\to Y\ni Q$ be an irreducible extremal neighborhood that contracts a divisor to a curve $\Gamma$.
If $Q\in \Gamma$ is a non-Gorenstein singularity, then $Q \in Y$ can not be of type $cE/2$,  $cD/3$ nor $cAx/4$.
\end{prop}

\begin{proof}
Denoted by $E$ a general elephant near $P\in Y$.

Suppose first that $Q \in Y$ is of type $cE/2$. Since the dual graph $\Delta(E)$ is of type $E_7$ by Table \ref{ta:table1}, it follows from Lemma \ref{deform} that every extremal neighborhood must be of type $2.2.1'.3$. By Lemma \ref{comp_d}, one sees that $\mu_X \ge 4$, which is impossible.

Similarly, if $Q \in Y$ is of $cD/3$, then the dual graph $\Delta(E)$ is of type $E_6$ by Table \ref{ta:table1}. It follows from Lemma \ref{deform} that every extremal neighborhood must be of type $2.2.1.2$, $2.2.1'.3$ or $2.2.2'$.
By Lemma \ref{comp_d}, one sees that $\mu_X \ge 6$, which is impossible.

Finally, if $ Q \in Y$ is of type $cAx/4$, then the dual graph $\Delta(E)$ is $D$-type. It follows from Lemma \ref{deform} that every extremal neighborhood can not be of type $2.2.1.1$ nor $2.2.4$. Therefore, each non-Gorenstein singularity on $X$ has index $2$, $4$ or an odd integer $m\geq 3$.
Taking a resolution over $X$ and computing the discrepancies over $Y$, one sees that each discrepancy
$a(F_i,Y)$ is of the form $\frac{a_i+\alpha_i}{2}, \frac{b_j+\beta_j}{4}$ or $\frac{c_l+\gamma_l}{m}$. None of these expression could be $\frac{1}{4}$, which is impossible.
\end{proof}

Notice that if the extremal neighborhood $X\supset C$ is semistable, that is $\Delta(E_Y)$ is $A$-type, then $Q\in Y$ must be of type $cA/r'$ by Lemma \ref{deform}.

From the classification of extremal neighborhood in Table \ref{ta:table2}, we have the computation (easier case). 

\begin{prop}\label{cAxcD}
Let $f\colon X\supset C\to Y\ni Q$ be an irreducible extremal neighborhood that contracts a divisor to a curve $\Gamma$.
If $Q\in Y$ is of type $cAx/2$ or $cD/2$, 
Then $\Xi(X) \ge \Xi(Y)$ and $F(X) > F(Y)$.
\end{prop}
\begin{proof}
Suppose $Q \in Y$ is of type $cAx/2$. By Lemma \ref{comp_d}, 
there exists at least one singularity of index greater or equal to 4 in the extremal neighborhood $X\supset C$, so $\Xi(X) \ge 4=\Xi(Y)$ and $F(X) > 3=F(Y)$ by Table~\ref{ta:table1}. 

Suppose that $Q \in Y$ is of type $cD/2$. 
Since the general elephant of $Q\in Y$ is $D_{2k'}$ where $k'$ is the axial weight, the extremal neighborhood can not be of type $2.2.1.1$ nor $2.2.4$.
By Lemma \ref{comp_d}, $\mu_{X \supset C} \geq 4$, so we don't need to consider the extremal neighborhood of type $2.2.1.2$, $2.2.1'.1$, $2.2.1'.2$, $2.2.1'.3$ and $2.2.5$.

\noindent
{\bf Case 1.} (In case $2.2.1.3$) $IIA$ ($cAx/4$), $\Delta(E_X) = \Delta(E_Y) = D_{2k+1}$. \\
We have $2k' \leq 2k+1$, hence $\Xi(Y)=2k' < 2k+2 = \Xi(X)$ and
\[
  F(Y)
  = \frac{3k'}{2}
  < \frac{6k+9}{4}
  = F(X).
\]
\noindent
{\bf Case 2.} (In case $2.2.1'.4$) $IIA+III$ ($cAx/4$), $\Delta(E_X) = \Delta(E_Y) = D_{2k+1}$. \\
This case is the same as the Case 1.

\noindent
{\bf Case 3.} (In case $2.2.2$)  $IC$ (cyclic quotient), $\Delta(E_X) = A_{m-1}, \Delta(E_Y) = D_{m}$ and $m\geq 5$ is odd. \\
We have $2k' \leq m$, so $ \Xi(Y)=2k' \leq m=\Xi(X)$ and
\[
  F(Y)= \frac{3k'}{2} < m - \frac{1}{m}=F(X).
\]

\noindent
{\bf Case 4.}  (In case $2.2.2'$) $IIB$ ($cAx/4$ with $k=2$), $\Delta(E_X) = D_{5}, \Delta(E_Y) = E_{6}$. \\
We have $2k'\leq 6$, so $ \Xi(Y)=2k'\leq 6=\Xi(X)$ and
\[
  F(Y)= \frac{3k'}{2} <  \frac{6\cdot 2+9}{4}=F(X).
\]

\noindent
{\bf Case 5.}  (In case $2.2.3$)  $IA+IA$, $\Delta(E_X) = A_{m-1}+D_{2k}, \Delta(E_Y) = D_{2k+m}$. Note that $m\geq 3$ is odd. \\
We have $2k' \leq 2k+m$, so $ \Xi(Y)=2k'\leq 2k+m=\Xi(X)$ and
\[
  F(Y)
  = k' \frac{3}{2}
  \leq k \frac{3}{2} + m \frac{3}{4}
  < k \frac{3}{2} + m - \frac{1}{m}
  = F(X).
\]

\noindent
{\bf Case 6.}  (In case $2.2.3'$) $IA+IA+III$, $\Delta(E_X) = A_{m-1}+A_1, \Delta(E_Y) = D_{m+2}$ where $m\geq 3$ is odd. \\
Since $2k' \leq m+2$, we have $\Xi(Y) = 2k' \leq   m+2 \leq \Xi(X)$. Also
\[
  F(Y)
  = k' \frac{3}{2}
  \leq  \frac{3m+6}{4}
  < \frac{3}{2} + m - \frac{1}{m}
  = F(X).
\]

\end{proof}

The following computations are similar to the previous case $cD/2$.

\begin{prop}\label{cA}
Let $f\colon X\supset C\to Y\ni Q$ be an irreducible extremal neighborhood that contracts a divisor to a curve $\Gamma$.
If $Q \in Y$ of type $cA/r'$, then $\Xi(X) \ge \Xi(Y)$ and $F(X) > F(Y)$.
\end{prop}
\begin{proof}
Let $k'$ be the axial weight of $Q$.

Suppose $E_X$ is not $A$-type.
Since $\mu_{X \supset C} \geq 4$,
the extremal neighborhood is not of type $2.2.1.2$, $2.2.1'.1$, $2.2.1'.2$, $2.2.1'.3$, and $2.2.5$.

\noindent
{\bf Case 1.} (In case $2.2.1.3$) $IIA$ ($cAx/4$), $\Delta(E_X) = \Delta(E_Y) = D_{2k+1}$. \\
We have $r'k'-1 < 2k+1$.  By Lemma~\ref{comp_d}, we see $2r' \leq 4$. So $r'=2$. It follows that $k' \leq k$ and $\Xi(Y)=2k'<2k+2=\Xi(X)$. Moreover,
\[
  F(Y)
  = \frac{3k'}{2}
  < \frac{6k+9}{4}
  = F(X).
\]

\noindent
{\bf Case 2.} (In case $2.2.1'.4$)  $IIA+III$ ($cAx/4$), $\Delta(E_X) = \Delta(E_Y) = D_{2k+1}$. \\
This case is the same as the Case 1.

\noindent
{\bf Case 3.} (In case $2.2.2$)  $IC$ (cyclic quotient), $\Delta(E_X) = A_{m-1}, \Delta(E_Y) = D_{m}$ and $m\geq 5$ is odd. \\
We have $r'k'-1 \leq m-1$. One has $\Xi(X)=r'k' \leq m=\Xi(Y)$ and $F(Y)=r'k'-\frac{k'}{r}< m-\frac{1}{m}=F(X)$.

\noindent
{\bf Case 4.} (In case $2.2.2'$) $IIB$ ($cAx/4$ with $k=2$), $\Delta(E_X) = D_{5}, \Delta(E_Y) = E_{6}$. \\
We have $r'k'-1\leq 6$. Since $\mu_{X\supset C}=4$, by Lemma \ref{comp_d}, we see that 
$r'=2$ and $k' \leq 3$.
Hence $ \Xi(Y)=2k'\leq 6=\Xi(X)$ and
\[
  F(Y)= \frac{3k'}{2} <  \frac{6\cdot 2+9}{4}=F(X).
\]


\noindent
{\bf Case 5.} (In case $2.2.3$) $IA+IA$, $\Delta(E_X) = A_{m-1}+D_{2k}, \Delta(E_Y) = D_{2k+m}$. Note that $m\geq 3$ is odd.

Now $r'k'-1 < m+2k$.
Since there exists an exceptional divisor over $Y$ with discrepancy $1/r$, we see 
$m=r'l$ for some positive integer $l$. Since $m$ is odd, so are $r'$ and $l$. 
Suppose $k=1$. If $r'k'= m+2k$, then $r'k'=m+2=r'l+2$ and hence $r'=1$.
If $r'k'<m+2k=m+2$, then 
\[F(Y)=r'k'-\frac{k'}{r'}\leq m+1- \frac{k'}{r'}< m -\frac{1}{m}+\frac{3}{2}=F(X).\]
So we may assume that $k\geq 2$.
If $r'=2$, then $\Xi(Y)=2k'\leq m+2k=\Xi(X)$ and
\[F(Y)=\frac{3}{2}k' \leq \frac{3}{2}k+\frac{3}{4}m< \frac{3}{2}k+m -\frac{1}{m} =F(X).\]

Suppose $r'\geq 3$. From Proposition \ref{flipcEcDcAx}, $Q\in Y$ is of type $cA/r'$.
\begin{claim}
$\Xi (Y) \leq m+2$.
\end{claim}
\begin{proof}

Put $f=f_0$, $X=X_0$, $Y=Y_0$, $\Gamma=\Gamma_0$ and $Q=Q_0$ and $F$ the exceptional divisor for $f$.
By \cite[Theorem 3.3]{CH11}, there exists a smallest positive integer $n$ such that for each $j=0,1,...,n-1$, we have the following factorization

\[
\xymatrix{
W_j\ar@{-->}[rr]^{\theta_j} \ar[d]_{g_j} & &X_{j+1}\ar[d]^{f_{j+1}} \\
X_j \ar[d]_{f_j}&& Y_{j+1}\ar[lld]^{k_{j+1}}\\
Y_j&},
\]
where $g_j$ is a $w$-morphism, $\theta_j$ is a composition of flips and probably a flop, $f_{j+1}$ contracts the proper transform $F_{X_{j+1}}$ to the curve $\Gamma_{j+1}$, $k_{j+1}$ is a divisorial contraction to a point $Q_{j}$, and the point $Q_n\in \Gamma_n$ is Gorenstein. 
Denote by $G^j$ the exceptional divisor of $g_j$. Then the proper transform $G^j_{Y_{j+1}}$ is the exceptional divisor of $k_{j+1}$.

Since $Q$ is of type $cA/r'$, by Kawakita's classification in \cite[Theorem 1.2]{Kawakita2}, 
each $k_{j+1}$ is a weighted blow up and there exists at most three non-Gorenstein singularities on the exceptional divisor $G^j_{Y_{j+1}}$ where at most one singularity is not cyclic quotient. 
Since $Q_{t+1}\in G^t_{Y_{t+1}}\cap \Gamma_{t+1}$ for every $t=0,...,n-2$,
$Q_{t+1}$ cannot be a cyclic quotient singularity by Kawamata in \cite{Kaw96}.
In particular, each $Q_{t+1}$ is of type $cA/r'$.


Let $E\in |-K_X|$ denote a general elephant of the extremal neighborhood $X\supset C$.
By our construction and \cite[Lemma 2.7]{CH11},
we see that for every $j=0,...,n-1$, $E_{W_j}\in |-K_{W_j}|$, $E_{X_{j+1}}\in |-K_{X_{j+1}}|$ and hence $E_{Y_{j+1}}\in |-K_{Y_{j+1}}|.$ 
Denote by $P_1\in C$ the singularity of index two on the extremal neighborhood $X\supset C$. By \cite[Chap 13, Appendix]{KM92} and Proposition \ref{k2Asing}, 
there must exist singularities $P_1,...,P_s$ satisfying all of the following conditions.
\begin{enumerate}
\item $s\geq 1$ and each $P_i$ has index $2$.
\item for every $i=1,...,s-1$, there is a flipping or flopping curve $C_i$ in $\theta_1,...,\theta_n$ with $P_i\in C_i$ and $P_{i+1}\in C_i^+$. 
\item  if $C_i$ is in the case 2.2.4, the axial weight of $P_{i+1}$ is larger or equal to that of $P_i$.
\item 
if $C'$ is a flipping curve containing $P_s$, then there is no non-Gorenstein point on flipped curve $C'^+$.
\end{enumerate}
Suppose 
$P_s$ is contained in the fiber of $f_{n}$.
As $E_{Y_{n}}$ is a partial resolution of $E_Y$, by Corollary \ref{exclude}, it follows that $\Xi (Y)\leq \Xi (Y_{n})\leq m$.
Otherwise, we may further assume that there is no flipping curve containing $P_s$.
So there exists a positive integer $j_0\leq n$ such that $P_s\in G^{j_0-1}_{X_{j_0}}$ and $P_s\not\in f_{j_0}^{-1}(Q_{j_0})$.
In particular, $P_s$ is the cyclic quotient $\frac{1}{2}(1,1,1)$.
As $E_{Y_{j_0}}$ is a partial resolution of $E_Y$, by Corollary \ref{exclude}, it follows $\Xi (Y)\leq \Xi (Y_{j_0})\leq m+2$.

\end{proof}
Therefore we have 
\[F(Y)=r'k'-\frac{k'}{r'}\leq m+2- \frac{k'}{r'} <  m -\frac{1}{m}+\frac{3}{2}k  =F(X).\]

\noindent
{\bf Case 6.} (In case $2.2.3'$) $IA+IA+III$, $\Delta(E_X) = A_{m-1}+A_1, \Delta(E_Y) = D_{m+2}$ where $m\geq 3$ is odd. \\ 
Now $r'k'-1 \leq m+1$, so we have $\Xi(Y)=r'k'\leq m+2=\Xi(X)$.
Suppose that  $r'k'  \leq m+1$, then clearly,
\[
  F(Y)=k' \paren{r'-\frac{1}{r'}} < m+1 \leq m - \frac{1}{m} + \frac{3}{2}=F(X).
\]
Suppose that $r'k' = m+2$.
From definition \ref{etnbd} and Lemma~\ref{comp_d}, we see that $f^{-1}(Q)=C$, 
and $r' \mid 2m$. It follows that $r' \mid 4$. Since $m$ is odd, $r'=2$.  In this situation,
\[
  F(Y)
  = \frac{3k'}{2}
  = \frac{3(m+2)}{4}
  < F(X).
\]

\noindent
Suppose now $E_X$ is of type $A$.

\noindent
{\bf Case 7.} (In case $2.2.1.1$) $cA/m$, $\Delta(E_X) = \Delta(E_Y) = A_{rk-1}$. \\
Suppose that $Q \in Y$ is a point of index $r'$ with axial weight $k'$, then $r' < r$ and $r'k' \leq rk$.
Hence $\Xi(Y) = r'k' \leq rk = \Xi(X)$.

If $k \geq k'$, together with $r' < r$, then we have
\begin{align*}
&F(Y) - F(X)=k' \paren{r'-\frac{1}{r'}} - k \paren{r-\frac{1}{r}}\\
&<k' \paren{r'-\frac{1}{r'}} - k \paren{r'-\frac{1}{r'}}=(k'-k) \paren{r'-\frac{1}{r'}}\leq 0.
\end{align*}
We may assume that $k < k'$.
Then $r'k \leq rk < rk'$ and so
\begin{align*}
  F(Y )-F(X)=k' \paren{r'-\frac{1}{r'}} - k \paren{r-\frac{1}{r}}
  &\leq  -\frac{k'}{r'}+\frac{k}{r}< 0.
\end{align*}


\noindent
{\bf Case 8.} (In case $2.2.4$) semistable $IA+IA$, $\Delta(E_X) = A_{r_1k_1-1}+A_{r_2k_2-1}, \Delta(E_Y) = A_{r_1k_1+r_2k_2-1}$. \\
Suppose that $Q \in Y$ is a point of index $r'$ with axial weight $k'$, then $\Xi(Y) =r'k' \leq r_1k_1+r_2k_2=\Xi(X)$.
This is similar to Case 7.
From \cite[Theorem 4.5]{Mori02}, one sees $r'=\gcd (r_1,r_2)$. Together with Lemma \ref{comp_d}, we have $r'\leq \min \{r_1,r_2\}$ and $r_1\neq r_2$.
If $k_1+k_2> k'$, then
\begin{align*}
F(Y) - F(X)&=k' \paren{r'-\frac{1}{r'}} - k_1 \paren{r_1-\frac{1}{r_1}}- k_2 \paren{r_2-\frac{1}{r_2}}\\
&< \paren{k_1+k_2} \paren{r'-\frac{1}{r'}} - k_1 \paren{r_1-\frac{1}{r_1}}- k_2 \paren{r_2-\frac{1}{r_2}}\\
&=k_1\paren{r'-\frac{1}{r'}-r_1+\frac{1}{r_1}}+k_2\paren{r'-\frac{1}{r'}-r_2+\frac{1}{r_2}}\leq 0.
\end{align*}
We may assume that $k_1+k_2\leq  k'$.
Then $$\frac{k_1}{r_1}+\frac{k_2}{r_2}< \frac{k_1}{r'}+\frac{k_2}{r'}\leq \frac{k'}{r'},$$
so
\begin{align*}
F(Y) - F(X)&=k' \paren{r'-\frac{1}{r'}} - k_1 \paren{r_1-\frac{1}{r_1}}- k_2 \paren{r_2-\frac{1}{r_2}}\\
&= \paren{r'k'-r_1k_1-r_2k_2}+\paren{\frac{k_1}{r_1}+\frac{k_2}{r_2}-\frac{k'}{r'}}< 0.
\end{align*}

\end{proof}

Combining Propositions \ref{cEcDcAx}, \ref{cAxcD}, and \ref{cA}, we obtain Theorem \ref{typeIEF} which provides a partial answer to Question \ref{FtypeI}. 
\begin{thm}\label{typeIEF}
Let $f\colon X\supset C\to Y\ni Q$ be an irreducible extremal neighborhood that contracts a divisor to a curve $\Gamma$.
Then $\Xi(X) \ge \Xi(Y)$ and $F(X) > F(Y)$, where $\Xi(X)$ denotes the integer $\sum_i r_i$ in Theorem \ref{KaRe}.

\end{thm}

\end{document}